\documentclass[10pt]{amsart}

\usepackage{amsmath,amssymb}
\usepackage{a4wide}
\usepackage{hyperref}
\usepackage{colortbl}
\usepackage{tikz}
\usepackage{tabularx}
\usepackage{chngpage}

\usepackage{graphicx}
\usepackage{xypic}
\entrymodifiers={+!!<0pt,\fontdimen22\textfont2>}
\usepackage[all]{xy}

\newtheoremstyle{myremark} 
    {\topsep}                    
    {\topsep}                    
    {}                           
    {}                           
    {\bf}                        
    {.}                          
    {.5em}                       
    {}  

\theoremstyle{plain}
\newtheorem{lemma}{Lemma}[section]
\newtheorem{theorem}[lemma]{Theorem}
\newtheorem{fact}[lemma]{Fact}
\newtheorem{definition}[lemma]{Definition}
\newtheorem{corollary}[lemma]{Corollary}
\newtheorem{proposition}[lemma]{Proposition}

\theoremstyle{myremark}
\newtheorem{remark}[lemma]{Remark}
\newtheorem{example}[lemma]{Example}

\newcommand{\zet}{\mathbb{Z}}

\newcommand{\niceu}{\mathcal{U}}

\newcommand{\dist}{\mathrm{dist}}

\newcommand{\bary}{\mathrm{bd}}

\newcommand{\disunion}{\sqcup}
\newcommand{\cl}{\mathrm{Cl}}
\newcommand{\ind}{I}
\newcommand{\compl}[1]{\overline{#1}}

\newcommand{\htpyequiv}{\simeq}
\newcommand{\homeo}{\equiv}
\newcommand{\susp}{\Sigma}
\newcommand{\simpjoin}{\ast}

\newcommand{\incl}{\hookrightarrow}
\newcommand{\x}[2]{\bigvee^{#1}S^{#2}}
\newcommand{\neib}[1]{N(#1)}
\newcommand{\st}{\mathrm{st}}
\newcommand{\lk}{\mathrm{lk}}
\newcommand{\te}{\mathrm{t}}
\newcommand{\supress}[1]{}
\renewcommand{\subset}{\subseteq}
\renewcommand{\supset}{\supseteq}

\newcommand\xx[1]{\cellcolor[gray]{0.82} {#1}}
\newcommand\nnn[2]{\tikz \node (#2) {#1};}
\tikzstyle{every picture}+=[remember picture,baseline]
\tikzstyle{every node}+=[inner sep=0pt,anchor=base,text depth=.25ex,outer sep=1.5pt]
\tikzstyle{every path}+=[dashed]

\begin{document}

\title{Clique complexes and graph powers}

\author[Micha{\l} Adamaszek]{Micha{\l} Adamaszek}
\address{Mathematics Institute and DIMAP,
      \newline University of Warwick, Coventry, CV4 7AL, UK}
\email{aszek@mimuw.edu.pl}
\thanks{Research supported by the Centre for Discrete
        Mathematics and its Applications (DIMAP), EPSRC award EP/D063191/1.}
\begin{abstract}
We study the behaviour of clique complexes of graphs under the operation of taking graph powers. As an example we compute the clique complexes of powers of cycles, or, in other words, the independence complexes of circular complete graphs.
\end{abstract}
\maketitle

\section{Introduction}

There are many constructions that associate topological spaces to graphs and a lot of work has gone into studying how they reflect the underlying graph theory. In this paper we look at clique complexes and their interaction with powers of graphs.

All our graphs are simple, finite and undirected. If $G$ is a graph and $r$ is a non-negative integer then the \emph{$r$-th power} or \emph{$r$-th distance power} of $G$, denoted $G^r$, is a new graph with the same vertex set in which two vertices are adjacent if and only if their distance in $G$ is \emph{at most} $r$. Any graph $G$ gives rise to a sequence of graph inclusions
\begin{equation}
G\incl G^2\incl G^3\incl\cdots
\end{equation}
which eventually stabilizes (at the complete graph if $G$ is connected).

For any graph $G$ the \emph{clique complex} $\cl(G)$ is a simplicial complex whose vertices are the vertices of $G$ and the simplices are the \emph{cliques} (complete subgraphs) in $G$. Clearly $\cl$ is a functor from graphs to simplicial complexes and we have inclusions
\begin{equation}
\label{seqinj}
\cl(G)\incl \cl(G^2)\incl \cl(G^3)\incl\cdots
\end{equation}
which, for a connected graph $G$, stabilize at the full simplex. In a geometer's language $\cl(G^r)$ is precisely the Vietoris-Rips complex whose faces are subsets of diameter at most $r$ in the discrete metric space $V(G)$ with the shortest path distance.

Note that not every graph is of the form $G^r$ for $r\geq 2$ (in fact already the recognition of graph squares is NP-hard, \cite{Mot}), so we may ask about interesting properties of the spaces $\cl(G^r)$ and of the inclusions $\cl(G^r)\hookrightarrow\cl(G^{r+1})$.

For example, if $G=C_7$ is the $7$-cycle then $\cl(C_7^2)$ has maximal faces of the form $\{i,i+1,i+2\}\pmod{7}$. It is homeomorphic to the M\"obius strip and it collapses to its subcomplex $\cl(C_7)\homeo S^1$. If, on the other hand, $G=C_6$, then the complex $\cl(C_6^2)$ is the boundary of the octahedron,  homeomorphic to $S^2$, and the sequence (\ref{seqinj}) is, up to homotopy, $S^1\to S^2\to \ast\to\cdots$.

Let us outline the structure of the paper. Section \ref{sec:largegirth} contains some preliminary results, in particular on powers of graphs with no short cycles. In Section \ref{sec:nobadcliques} we restrict to graph squares ($r=2$) and prove topological and combinatorial conditions which guarantee that the inclusion $\cl(G)\incl\cl(G^2)$ is a homotopy equivalence. 

In Section \ref{sec:realization} we discuss \emph{universality} of $\cl(G^r)$, proving that for any $r$ every finite complex can be realized as $\cl(G^r)$ up to homotopy. Contrary to the case $r=1$, for higher $r$ not every space has a realization as $\cl(G^r)$ up to homeomorphism. Our method is based on some results of \cite{Doch} and the analysis of shortest paths in iterated barycentric subdivisions.  

Section \ref{sect:line} provides a complete description of the clique complexes of the total graph and the line graph of $G$. 

In the last part, Section \ref{sec:cyclepowers}, we calculate the homotopy types of $\cl(G^r)$ in the first nontrivial case, that is for the cycles $G=C_n$. A quick preview of those can be found in Section \ref{sect:tableofcyc}. They turn out to be obtained from a small number of initial cases by an action of a double suspension operator $\susp^2$. To see this we run the theory of star clusters of \cite{Bar} on the independence complexes of the complements $\compl{C_n^r}$, the circular complete graphs.

\subsection*{Acknowledgement.} The referee's insightful remarks were very helpful in improving the quality of the paper.

\subsection*{Notation.} We follow standard notation related to graphs and (combinatorial) algebraic topology. Let us just fix a few conventions.

\subsection*{Graphs.} We write $V(G)$ and $E(G)$, respectively, for the set of vertices and edges of an undirected, simple graph $G$. Given a vertex $v$ of $G$ we define the neighbourhood $N_G(v)=\{w: vw\in E(G)\}$ and the closed neighbourhood $N_G[v]=N_G(v)\cup\{v\}$. We write $\dist_G(u,v)$ for the length of the shortest path in $G$ from $u$ to $v$. A graph is called a \emph{cone} if there is a vertex $v$ adjacent to every other vertex, i.e. $N_G[v]=V(G)$. By convention the $0$-th power $G^0$ is the graph with vertex set $V(G)$ and no edges.

The \emph{girth} of a graph is the length of its shortest cycle or $\infty$ for a forest. The symbol $\compl{G}$ denotes the complement of $G$ and $\disunion$ is the disjoint union of graphs. If $W\subset V(G)$ then $G[W]$ denotes the subgraph of $G$ induced by $W$. The symbols $K_n$, $C_n$, $P_n$ denote, respectively, the complete graph, cycle and path with $n$ vertices.

\subsection*{Simplicial topology.}
If $G$ is a graph then $\ind(G)$ and $\cl(G)$ denote, respectively, the \emph{independence complex} and the \emph{clique complex} of $G$. They both have $V(G)$ as vertex set and the faces are, respectively, the independent sets or the cliques in $G$.  Clearly $\ind(G)=\cl(\compl{G})$. 

If $K$, $L$ are simplicial complexes with disjoint vertex sets then the \emph{join} $K\simpjoin L$ is the simplicial complex whose faces are all the unions $\sigma\cup\tau$ for $\sigma\in K$ and $\tau\in L$. The \emph{cone} $CK$ is the join of $K$ with one point (the apex) and the \emph{unreduced suspension} is $\susp K=S^0\simpjoin K$. The symbol $\sqcup$ is the disjoint union. By $K^{(n)}$ we denote the $n$-dimensional skeleton of $K$. Every graph can be treated as a $1$-dimensional simplicial complex.

We also have induced subcomplexes. If $W$ is a subset of the vertices of $K$ then $K[W]$ denotes the subcomplex induced by $W$, i.e. the simplicial complex with vertices $W$ whose faces are all the faces of $K$ contained in $W$. We have the isomorphisms
\begin{center}
\begin{tabular}{ll}
$\cl(G\disunion H)=\cl(G)\disunion\cl(H)$ & $\ind(G\disunion H)=\ind(G)\simpjoin\ind(H)$\\
$\cl(G[W])=\cl(G)[W]$                     & $\ind(G[W])=\ind(G)[W]$.
\end{tabular}
\end{center}
We write $\bigvee^kX$ for a wedge of $k$ copies of a topological space $X$. The symbol $\homeo$ means homeomorphism and $\htpyequiv$ stands for homotopy equivalence. We do not distinguish between a simplicial complex and its geometric realization. 

We are going to use the following elementary language of discrete Morse theory to describe collapsing sequences (see \cite[Chapter 11]{Koz}, \cite{For1,For2}).
\begin{definition}
\label{acyclic}
An \emph{acyclic matching} in a simplicial complex $K$ is a set $M\subset K\times K$ of pairs of faces such that 
\begin{itemize}
\item if $(\sigma,\tau)\in M$ then $\sigma$ is a codimension $1$ face of $\tau$,
\item every face $\sigma$ belongs to at most one element of $M$,
\item there is no cycle
$$\sigma_0,\tau_0,\sigma_1,\tau_1,\sigma_2,\ldots,\sigma_n,\tau_n,\sigma_0$$
such that $(\sigma_i,\tau_i)\in M$, $\sigma_{i+1}$ is a codimension $1$ face of $\tau_i$ (where $\sigma_{n+1}=\sigma_0$), all $\sigma_i$ are distinct and $n\geq 1$.	
\end{itemize}
The faces of $K$ which do not belong to any element of $M$ are called \emph{critical}.
\end{definition}
\begin{fact}
If $K$ is a simplicial complex with an acyclic matching whose set of critical faces is a non-empty simplicial subcomplex $L$ then $K$ simplicially collapses to $L$.
\end{fact}
For other standard notions of combinatorial topology see \cite{Koz,Bjorner}.

\section{Preliminaries}
\label{sec:largegirth}

\begin{fact}
\label{firstfact}
For any connected graph $G$ the map of fundamental groups 
$$\pi_1(\cl(G))\to\pi_1(\cl(G^r))$$
induced by the inclusion $G\hookrightarrow G^r$ is \emph{surjective}.
\end{fact}
\begin{proof}
It suffices to prove that $\pi_1(\cl(G^{r-1}))\to\pi_1(\cl(G^{r}))$ is surjective for $r\geq 2$. Consider a based path $\alpha$ in $\cl(G^{r})$. By cellular approximation we can assume it lies in the $1$-skeleton and is piecewise linear. If $e=uv\in E(G^r)\setminus E(G^{r-1})$ then there is a vertex $w$ such that $uw,wv\in E(G^{r-1})$. Then $\{u,w,v\}$ is a face of $\cl(G^r)$ and any segment of $\alpha$ going along $uv$ can be continuously deformed to go along $uwv$ without moving the endpoints. Performing this operation for every segment in $E(G^r)\setminus E(G^{r-1})$ we obtain a  based path homotopic to $\alpha$ which lies in $\cl(G^{r-1})$.
\end{proof}

One situation when $\cl(G^r)$ is homotopy equivalent to (in fact, collapses to) $\cl(G)$ is when $r$ is not too large compared to the girth of $G$. Of course as soon as $G$ is triangle-free $\cl(G)\homeo G$ is $1$-dimensional.
\begin{proposition}
\label{prop:largegirth}
Let $r\geq 1$. If $G$ is a graph of girth at least $3r+1$ then for every $2\leq k\leq r$ the complex $\cl(G^k)$ collapses to $\cl(G^{k-1})$. In particular $\cl(G^r)$ collapses to its subcomplex $\cl(G)\homeo G$.
\end{proposition}
\begin{proof}
Let $\mathcal{E}=E(G^k)\setminus E(G^{k-1})$ be the set of ``new'' edges in $G^k$ and let $\mathcal{F}\subset\cl(G^k)$ be the set of faces which contain at least one edge of $\mathcal{E}$. We have $\cl(G^{k-1})=\cl(G^k)\setminus\mathcal{F}$. If $\mathcal{E}=\emptyset$ there is noting to do, so assume $\mathcal{E}\neq\emptyset$.

The nonexistence of cycles of length $3k$ or less in $G$ has the following consequences. First, every maximal clique $\sigma$ in $G^k$ corresponds to a subtree of diameter $k$ in $G$. Second, every edge in $\mathcal{E}$ (hence also every face in $\mathcal{F}$) belongs to a \emph{unique} maximal face of $\cl(G^k)$. To see the second statement let $e=uv\in\mathcal{E}$ and suppose $x$, $y$ are two vertices such that $xuv$ and $yuv$ are both faces of $\cl(G^k)$. Denote by $\alpha$ the shortest path in $G$ from $u$ to $v$. By the first observation there is a vertex $x'\in\alpha$ such that the shortest paths from $x$ to $u$ and $v$ join the path $\alpha$ at $x'$. Similarly, there is a  $y'\in\alpha$ with the same property for $y$ and we may assume w.l.o.g. that the order along $\alpha$ is $u-x'-y'-v$. Then
\begin{eqnarray*}
\dist_G(x,y)&=&\dist_G(x,x')+\dist_G(x',y')+\dist_G(y',y)=\\
&=&\dist_G(x,v)+\dist_G(y,u)-\dist_G(u,v)\leq k+k-k=k
\end{eqnarray*}
which proves the claim.

Let $\sigma$ be some maximal face of $\cl(G^k)$ and $v\in\sigma$ any fixed vertex whose distance in $G$ to all vertices of $G[\sigma]$ is strictly less than $k$ (for example the centre of the tree $G[\sigma]$). We define an acyclic matching $M_\sigma$ on $\sigma$ by taking all the pairs $$(f,f\cup\{v\})$$ for all faces $f\in \mathcal{F}$ such that $f\subset \sigma\setminus\{v\}$. Since no edge of $\mathcal{E}$ which lies in $\sigma$ has $v$ as its endpoint, every face of $\mathcal{F}$ contained in $\sigma$ is of the form $f$ or $f\cup\{v\}$ above. It follows that $M_\sigma$ matches all faces of $\sigma$ which are in $\mathcal{F}$ (and only those).

Let $M=\bigcup_\sigma M_\sigma$ be the union of those matchings over all maximal faces $\sigma$. By the previous remarks it is well-defined, acyclic and its critical faces form the subcomplex $\cl(G^{k-1})$. This ends the proof.
\end{proof}

The girth bound of $3r+1$ is optimal, because $\cl(C_{3r}^r)\htpyequiv\bigvee^{r-1} S^2$ by the results of Section \ref{sec:cyclepowers}.

One standard way of analyzing the homotopy type of $\cl(G)$ is via the notions of folds and dismantlability. If $u,v\in V(G)$ are distinct vertices such that $N_G[u]\subset N_G[v]$ then we say $G$ \emph{folds onto} $G\setminus u$. A graph $G$ is \emph{dismantlable} if there exists a sequence of folds from $G$ to a single vertex. It is a classical fact that a fold preserves the homotopy type of the clique complex and, in fact, induces a collapsing of $\cl(G)$ onto $\cl(G\setminus u)$, so the clique complex of a dismantlable graph is collapsible (see for example \cite[Lemma 2.2]{BFJ}). In this context we have the following simple result.
\begin{lemma}
\label{lemma:dismant}
If $G$ is dismantlable then so is $G^r$ for any $r\geq 1$.
\end{lemma}
\begin{proof}
We use induction on $|V(G)|$. Let $u$ be a vertex such that $G$ folds onto $G\setminus u$ and $G\setminus u$ is dismantlable. Let $v$ be a vertex which satisfies $N_G[u]\subset N_G[v]$. First note that
$$(G\setminus u)^r=G^r\setminus u.$$
Indeed, the inclusion $\subset$ is obvious. For $\supset$ note that any occurrence of $u$ in a path can be replaced with $v$ or removed without increasing the length of the path.

The graph $G^r\setminus u=(G\setminus u)^r$ is dismantlable by induction. Moreover $N_{G^r}[u]\subset N_{G^r}[v]$. It follows that $G^r$ folds onto $G^r\setminus u$ and the dismantlability of $G^r$ is proved.
\end{proof}

Both \ref{prop:largegirth} and \ref{lemma:dismant} imply the following.
\begin{corollary}
\label{cor:treepower}
For every tree $T$ and any integer $r$ the complex $\cl(T^r)$ is collapsible (and, in particular, contractible).
\end{corollary}

\section{Stability}
\label{sec:nobadcliques}

In this section we only consider graph squares ($r=2$). We describe more general criteria which guarantee that the inclusion $\cl(G)\incl\cl(G^2)$ is a homotopy equivalence. 

Note that for any vertex $v$ of $G$ the set $N_G[v]$ forms a clique in $G^2$. 
\begin{theorem}
\label{thm:equiv}
Suppose $G$ satisfies the following condition:
\begin{itemize}
\item Every clique in $G^2$ is contained in a set of the form $N_G[v]$ for some vertex $v$.
\end{itemize}
Then the inclusion $i:\cl(G)\incl\cl(G^2)$ is a homotopy equivalence.
\end{theorem}
\begin{proof}
By passing to connected components we can assume $G$ is connected. 
We use the following local criterion of \cite[Thm. 6]{McC} (see also \cite[Cor. 1.4]{May}):
\begin{itemize}
\item Suppose $p:X\to Y$ is a continuous map and $Y$ has an open cover $\niceu=\{U_\alpha\}$ such that if $U,V\in\niceu$ then $U\cap V\in\niceu$. If for every $U\in\niceu$ the restriction $p|_{p^{-1}(U)}:p^{-1}(U)\to U$ is a weak equivalence then so is $p$.
\end{itemize}
Since we are working with finite simplicial complexes we can just as well replace open sets with closed subcomplexes (by taking a small open neighbourhood of a subcomplex) and weak equivalences with homotopy equivalences (by Whitehead's theorem).

For each $v\in V(G)$ let $U_v=\cl(G^2)[N_G[v]]$. Each of $U_v$ is a simplex in $\cl(G^2)$. By assumption we have $\cl(G^2)=\bigcup_{v\in V(G)} U_v$ and we can take a cover $\niceu$ of $\cl(G^2)$ consisting of all intersections of the sets $U_v$.


If $U=U_{v_1}\cap\cdots\cap U_{v_k}$ is non-empty then it is an intersection of faces of $\cl(G^2)$, hence it is contractible. It remains to show that $i^{-1}(U)$ is also contractible. Let $X=\bigcap_{i=1}^k N_G[v_i]$ be the set of vertices spanning $U$. Since $i$ is a subcomplex inclusion, we have $i^{-1}(U)=\cl(G)[X]$. Because in $G$ every vertex of $X$ is in distance at most $1$ from each of $v_i$, the set $X\cup\{v_1,\ldots,v_k\}$ forms a clique in $G^2$. Our assumption then gives a vertex $v$ such that
$$X\cup\{v_1,\ldots,v_k\}\subset N_G[v].$$
In particular $v\in N_G[v_i]$ for each $i=1,\ldots,k$, so $v\in X$. Moreover, since $X\subset N_G[v]$, the vertex $v$ is adjacent in $G$ to every other element of $X$, i.e. $G[X]$ is a cone with apex $v$. It implies that $\cl(G)[X]=\cl(G[X])$ is a simplicial cone with apex $v$, hence it is contractible. This completes the proof.
\end{proof}

\begin{figure}
\includegraphics[scale=0.15]{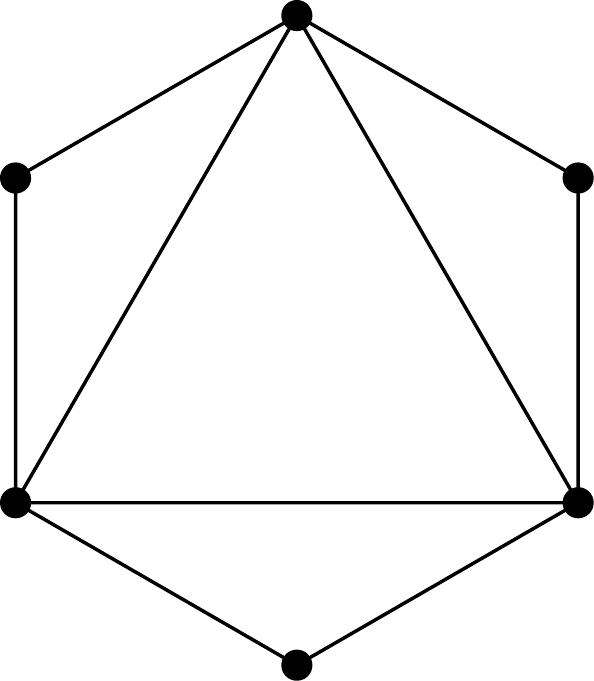}
\caption{The $3$-sun graph $S_3$.}
\label{fig:3sun}
\end{figure}

There is a more direct combinatorial condition which guarantees that the assumption of Theorem \ref{thm:equiv} is satisfied. Recall that we say $G$ is \emph{$H$-free} if $G$ does not have an \emph{induced} subgraph isomorphic to $H$. If $H_1,\ldots,H_k$ is a sequence of graphs then $G$ is $(H_1,\ldots,H_k)$-free if it does not have any of the $H_i$ as induced subgraphs. 

Consider the graph of Fig.\ref{fig:3sun}, usually denoted $S_3$ and called $3$-sun.
\begin{theorem}
\label{thm:induced}
If $G$ is $(C_4,C_5,C_6,S_3)$-free then $G$ satisfies the condition in Theorem \ref{thm:equiv}
\end{theorem}
\begin{proof}
Suppose, on the contrary, that $K=\{v_1,\ldots,v_k\}$ is the smallest clique in $G^2$ which is not contained in any set $N_G[v]$. Then $k\geq 3$ and there exist $w_1,\ldots,w_k$ such that $K\setminus v_i\subset N_G[w_i]$. The vertices $w_1,\ldots,w_k$ are pairwise distinct (as $w_i=w_j$ would mean $K\subset N_G[w_i]$) and there is no edge $w_iv_i$ in $G$ for any $i$ (same reason). It means that we have
\begin{equation}
\label{loc1}
\tag{*}\left\{
\begin{array}{l}
N_G[w_1]\cap\{v_1,v_2,v_3\}=\{v_2,v_3\}\\
N_G[w_2]\cap\{v_1,v_2,v_3\}=\{v_3,v_1\}\\
N_G[w_3]\cap\{v_1,v_2,v_3\}=\{v_1,v_2\}\\
v_1\neq v_2\neq v_3\neq v_1, \quad w_1\neq w_2\neq w_3\neq w_1
\end{array}\right.
\end{equation}
By Theorem 3 of \cite{BDS} a graph is $(C_4,C_5,C_6,S_3)$-free if and only if it does not have a configuration satisfying (\ref{loc1}).
\footnote{Which, using the notation of \cite{BDS}, is saying that the \emph{neighbourhood hypergraph} of $G$ is triangle-free.} That ends the proof.
\end{proof}

\begin{remark} For an arbitrary graph $G$ one might at least hope that the inclusion $\cl(G)\incl\cl(G^2)$ stabilizes the homotopy type, for example by increasing the connectivity of the space. This is not the case. For example, let $G$ be the graph of Fig.\ref{fig:nasty} and let $V$ denote the set of vertices of the outermost $6$-cycle. Then $\cl(G)$ is contractible while one can check by a direct calculation that $H_2(\cl(G^2))=\zet\oplus\zet$ where one of the generators is represented by the subcomplex $\cl(G^2)[V]$, homeomorphic to $S^2$.
\end{remark}
\begin{figure}
\includegraphics[scale=0.15]{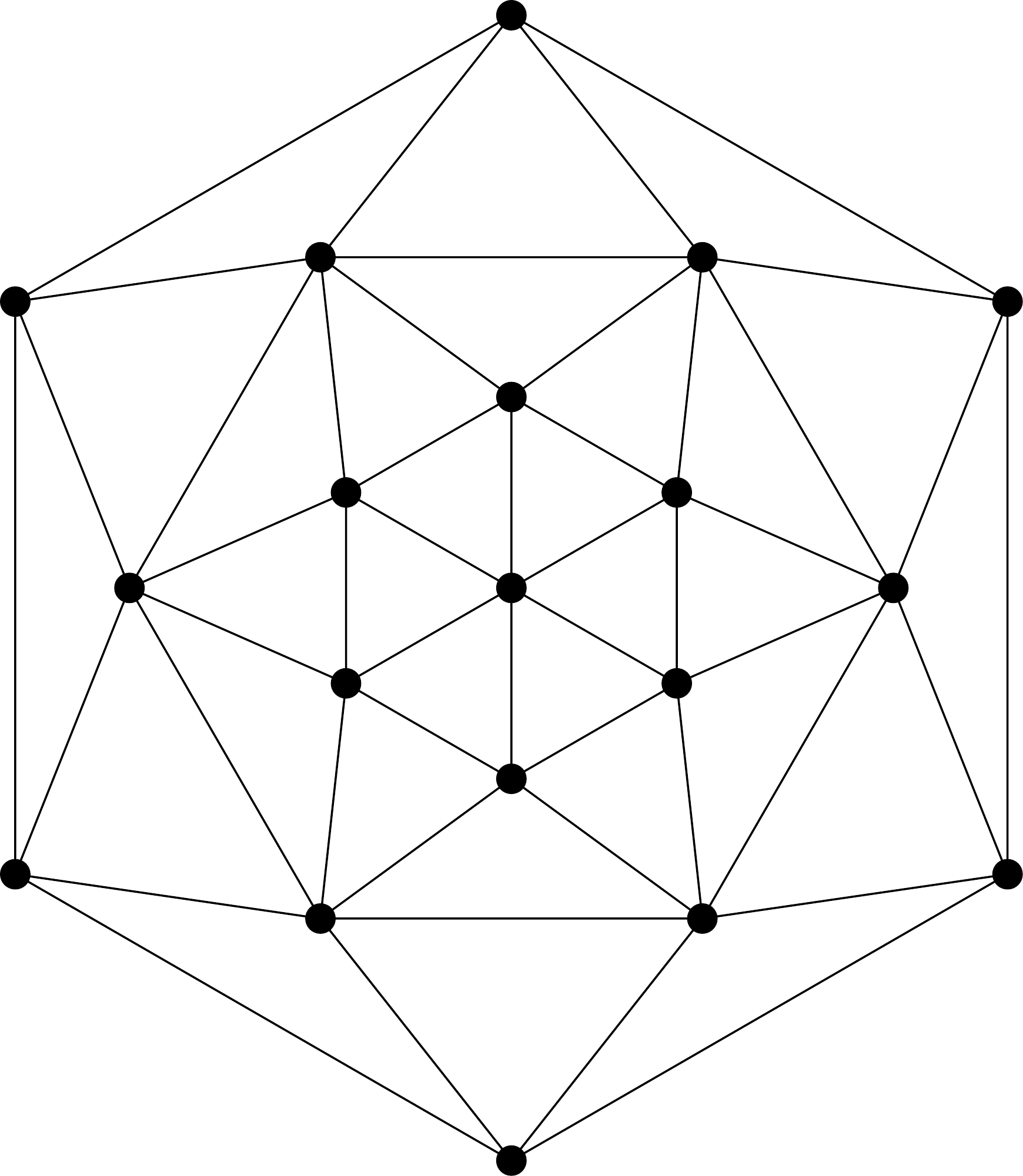}
\caption{A graph $G$ with $\cl(G)\htpyequiv\ast$ and $H_2(\cl(G^2))\neq 0$.}
\label{fig:nasty}
\end{figure}

\begin{remark}
The converse of Theorem \ref{thm:induced} is false as can be seen by taking any graph $G$ which is a cone and has one of the forbidden induced subgraphs. The converse of Theorem \ref{thm:equiv} is also false and the counterexample is the 3-sun $S_3$. Indeed, $\cl(S_3)\htpyequiv \cl(S_3^2)\htpyequiv\ast$ and $S_3^2=K_6$ is complete but $S_3$ itself is not a cone.
\end{remark}

\section{Universality}
\label{sec:realization}
It is a known fact that any finite simplicial complex $K$ is homeomorphic to $\cl(G)$ for some graph $G$. One can take $G$ to be the $1$-skeleton of the barycentric subdivision of $K$.

Clearly clique complexes of higher graph powers cannot represent all homeomorphism types. For instance, if $\cl(G^2)$ is a connected space of dimension two then every vertex of $G$ must have degree at most $2$. It means $G$ must be a path or cycle and a direct check narrows the possible two-dimensional homeomorphism types of $\cl(G^2)$ to $D^2$, $S^2$, the M\"obius strip and $D^1\times S^1$, where $D^n$ is the $n$-dimensional disk.

It is, however, true that arbitrary graph powers realize all homotopy types.
\begin{theorem}
\label{thm:realization}
For every finite simplicial complex $K$ and integer $r\geq 1$ there exists a graph $G$ such that $\cl(G^r)$ is homotopy equivalent to $K$.
\end{theorem}

In fact there is an explicit candidate for $G$. Given a finite complex $K$ and $s\geq 0$ let $\bary^s K$ denote its $s$-th iterated barycentric subdivision and let the graph $G_{s}$ be its $1$-skeleton: $$G_{s}=(\bary^s K)^{(1)}$$
(from now on we will suppress the complex $K$ from notation). Replacing, if needed, $K$ with its subdivision we can assume $K=\cl(G_0)$ and then for every $s\geq 0$ we have $\bary^s K=\cl(G_s)$. Then we have the following result.

\begin{theorem}
\label{thm:explicit}
For any finite simplicial complex $K$ and $1\leq r<2^{s-2}$
$$\cl((G_{s})^r)\htpyequiv K.$$
\end{theorem}

The proof strategy resembles that of \cite{Doch}. For any vertex $v$ of the original complex $K$ let $B_{s,v}$ and $S_{s,v}$ denote the vertex sets in $G_s$ defined as
\begin{eqnarray*}
B_{s,v}&=&\{w: \dist_{G_{s}}(v,w)< 2^s\},\\
S_{s,v}&=&\{w: \dist_{G_{s}}(v,w) = 2^s\}.
\end{eqnarray*}
The letters $B$ and $S$ stand for the open Ball and the Sphere of radius $2^s$ around $v$ in $G_s$. In the geometric realization the vertices of $B_{s,v}$ belong to the open star $\st_K(v)\setminus\lk_K(v)$ of $v$ in $K$ while the vertices of $S_{s,v}$ lie in the link $\lk_K(v)$. Note that $B_{s,v_1}\cap\cdots\cap B_{s,v_k}$ is nonempty if and only if $\{v_1,\ldots,v_k\}$ is a face of $K$.

The main technical result we use is proved in \cite[3.7,3.8]{Doch}.
\begin{proposition}[\cite{Doch}] 
\label{prop:doch}
For any face $\{v_1,\ldots,v_k\}$ of $K$ the graph
$$G_s[B_{s,v_1}\cap\cdots\cap B_{s,v_k}]$$
is dismantlable.
\end{proposition}

Now consider an integer $r<2^{s-2}$. We intend to prove that $\cl(G_s^r)\htpyequiv K$ using the nerve lemma \cite[15.21]{Koz}. Define subcomplexes of $\cl(G_s^r)$ by
\begin{equation}
\label{eq:defofx}
X_{s,v}=\cl((G_s[B_{s,v}])^r)\subset \cl(G_s^r)
\end{equation}
for the vertices $v$ of $K$. The reader should be warned that the subcomplex $X_{s,v}$ is \emph{not} induced; in particular it should not be confused with $\cl(G_s^r[B_{s,v}])$, which is usually bigger.

\begin{proposition}
\label{lemma:covering}
The family of subcomplexes $X_{s,v}$ is a covering of $\cl(G_s^r)$. The nerve of this covering is $K$.
\end{proposition}
\begin{proof}
The second statement is obvious since the vertex set of $X_{s,v_1}\cap\cdots\cap X_{s,v_k}$ is $B_{s,v_1}\cap\cdots\cap B_{s,v_k}$ and this is nonempty only for a face $\{v_1,\ldots,v_k\}$ of $K$.

Let us prove the first statement. Suppose $\sigma$ is a clique in $G_s^r$. Fix any $w\in\sigma$. There exists a vertex $v$ of $K$ such that 
$$\dist_{G_s}(v,w)\leq 2^{s-1}.$$
Fix also that $v$. Now any vertex $w'\in\sigma$ satisfies
\begin{eqnarray*}
\dist_{G_s}(w',v)&\leq& \dist_{G_s}(w',w)+\dist_{G_s}(w,v)\leq\\
&\leq& r+2^{s-1}<2^{s-2}+2^{s-1}\leq 2^s-1.
\end{eqnarray*}
Therefore $\sigma\subset B_{s,v}$.

Now we want to show that for any two vertices $w',w''\in\sigma$ the shortest path from $w'$ to $w''$ in $G_s$ lies in $G_s[B_{s,v}]$. Indeed, if $z$ is any vertex on that path then
\begin{eqnarray*}
\dist_{G_s}(z,v)&\leq& \dist_{G_s}(z,w')+\dist_{G_s}(w',v)\leq\\
&\leq& r+(r+2^{s-1})<2\cdot2^{s-2}+2^{s-1}= 2^s.
\end{eqnarray*}
so $z\in B_{s,v}$. Since $\sigma$ is a set of diameter at most $r$ in $G_s$ and the shortest paths between its vertices lie in $B_{s,v}$ it follows that $\sigma$ is a set of diameter at most $r$ in $G_s[B_{s,v}]$. It means that $\sigma\in X_{s,v}$.
\end{proof}

The point here was that the whole clique $\sigma$ was located at least $r$ steps away from $S_{s,v}$, so the path in $G_s$ could not take the advantage of any shortcut outside $B_{s,v}$.

Proposition \ref{lemma:covering} and the nerve lemma \cite[15.21]{Koz} imply Theorem \ref{thm:explicit} as soon as we prove that the nonempty intersections $X_{s,v_1}\cap\cdots\cap X_{s,v_k}$ are contractible. This is arranged for by the following lemma.
\begin{proposition}
\label{lemma:intersect}
For any vertices $v_1,\ldots,v_k$ of $K$ we have
$$X_{s,v_1}\cap\cdots\cap X_{s,v_k}=\cl((G_s[B_{s,v_1}\cap\cdots\cap B_{s,v_k}])^r)$$
\end{proposition}
\begin{proof}
We can restrict to the case when $\{v_1,\ldots,v_k\}$ is a face of $K$, otherwise the intersections are empty. By the definition of $X_{s,v}$ what we need to prove is
$$\cl((G_s[B_{s,v_1}])^r)\cap\cdots\cap \cl((G_s[B_{s,v_k}])^r)=\cl((G_s[B_{s,v_1}\cap\cdots\cap B_{s,v_k}])^r).$$
The inclusion $\supset$ is obvious, so we need to prove $\subset$. It is equivalent to the statement
\begin{itemize}
\item[$\mathcal{D}(k)\ $:] If $u,w\in B_{s,v_1}\cap\cdots\cap B_{s,v_k}$ are vertices such that
$$\dist_{G_s[B_{s,v_1}]}(u,w)\leq r,\ \ldots,\ \dist_{G_s[B_{s,v_{k}}]}(u,w)\leq r$$
then
$$\dist_{G_s[B_{s,v_1}\cap\cdots\cap B_{s,v_k}]}(u,w)\leq r.$$
\end{itemize}
We prove it by induction on $k$. Clearly $\mathcal{D}(1)$ holds. Now suppose $k\geq 2$. By induction there is a path $\alpha$ from $u$ to $w$ in $G_s[B_{s,v_1}\cap\cdots\cap B_{s,v_{k-1}}]$ of length at most $r$. Denote by $\beta$ the path in $G_s[B_{s,v_k}]$ from $u$ to $w$ of length at most $r$. If $\alpha$ lies completely in $B_{s,v_k}$ or $\beta$ lies in $B_{s,v_1}\cap\cdots\cap B_{s,v_{k-1}}$ then $\mathcal{D}(k)$ follows. If none of those two cases holds then $\alpha$ passes through some point $p\in S_{s,v_k}$ and $\beta$ passes through some $q\in S_{s,v_1}\cup\cdots\cup S_{s,v_{k-1}}$. Assume without loss of generality that $q\in S_{s,v_1}$. Then $\xymatrix{p \ar@{-}[r]^\alpha& u \ar@{-}[r]^\beta & q}$ is a path in $G_s$ of length at most $2r<2\cdot 2^{s-2}<2^s$ which connects $p\in B_{s,v_1}\cap S_{s,v_k}$ with $q\in  S_{s,v_1}\cap B_{s,v_k}$ and this whole path lies in $B_{s,v_1}\cup B_{s,v_k}$ (because $\alpha\subset G_s[B_{s,v_1}]$ and $\beta\subset G_s[B_{s,v_k}]$). The existence of such path, however, is excluded by the next lemma and this contradiction ends the inductive step.
\end{proof}

We are left with the last technical lemma whose intuitive meaning is the following. Suppose $\sigma,\tau$ are two faces of the same simplex in $K$. Suppose we look at the $s$-th barycentric subdivision of $K$ and the paths in its $1$-skeleton. Then the points of $\sigma$ are very far apart from the points of $\tau$ if one is not allowed to go through $\sigma\cap\tau$.

Using the standard notation
$$\dist_G(X,Y)=\min\{\dist_G(x,y): x\in X, y\in Y\}$$
for $X,Y\subset V(G)$ we can express this idea as follows (see Fig.\ref{fig:ugh}).
\begin{lemma}
\label{lemma:distbig}
For any two adjacent vertices $u,v$ of the original complex $K$
$$\dist_{G_s[B_{s,u}\cup B_{s,v}]}(B_{s,u}\cap S_{s,v}, S_{s,u}\cap B_{s,v})=2^s.$$
\end{lemma}
\begin{proof}
\begin{figure}
\includegraphics[scale=1]{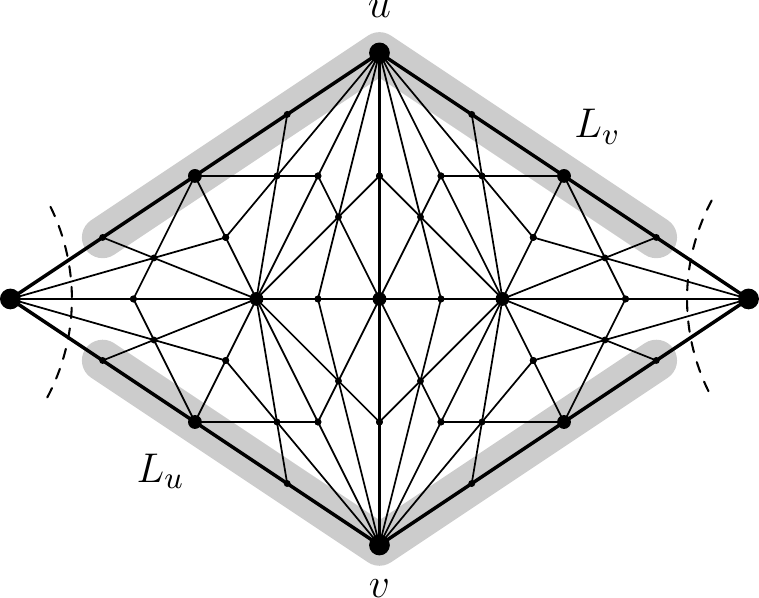}
\caption{An example with $s=2$. The shaded sets $L_u=S_{s,u}\cap B_{s,v}$ and $L_v=B_{s,u}\cap S_{s,v}$ contain vertices of $B_{s,u}\cup B_{s,v}$ in distance $2^s$ from, respectively, $u$ and $v$. The distance between these two sets within $B_{s,u}\cup B_{s,v}$ is also $2^s$, although their distance in $G_s$ is only $2$.}
\label{fig:ugh}
\end{figure}
A \emph{partial labeling} $l$ of a graph $G$ is an assignment of a real number $l(v)$ to some of the vertices of $G$. If $X\subset V(G)$ we write $l(X)$ for the set of labels assigned to the vertices in $X$, with $l(X)=\emptyset$ if the value of $l(v)$ is undefined for all $v\in X$.

We will construct partial labelings $l_s$ of $G_s$ for $s\geq 0$ with the properties:
\begin{itemize}
\item[a)] The set of vertices for which $l_s$ is defined is $B_{s,u}\cup B_{s,v}$.
\item[b)] For every simplex $\sigma\in\cl(G_s)$ the vertices of $\sigma$ are assigned at most two different labels.
\item[c)] For every edge $xy\in E(G_s)$ such that both $l_s(x)$ and $l_s(y)$ are defined we have $$|l_s(x)-l_s(y)|\in\{0,\frac{1}{2^s}\}.$$
\item[d)] $l_s(B_{s,u}\cap S_{s,v})=\{0\}$, $l_s(S_{s,u}\cap B_{s,v})=\{1\}$.
\end{itemize}

The partial labeling $l_0$ is defined by $l_0(u)=0$, $l_0(v)=1$ and undefined otherwise. Suppose $l_{s-1}$ has been defined. Every vertex $x\in V(G_s)$ represents a face $\tau\in\cl(G_{s-1})$ and we set
\begin{equation*}
l_s(x)=\left\{ 
\begin{array}{ll}
a & \mathrm{if}\ l_{s-1}(\tau)=\{a\} \\
(a+b)/2 & \mathrm{if}\  l_{s-1}(\tau)=\{a,b\} \\
\mathrm{undefined} & \mathrm{if}\  l_{s-1}(\tau)=\emptyset .
\end{array}
\right.
\end{equation*}
This is well-defined since $l_{s-1}$ satisfies b). Note that if $x\in V(G_{s-1})$ then $l_s(x)=l_{s-1}(x)$.

To prove that $l_s$ satisfies a) recall that $B_{s-1,u}\cup B_{s-1,v}$ are the vertices of $G_{s-1}$ located in the union of the open stars of $u$ and $v$ in $K$. Therefore a vertex $x$ of $G_s$ receives a label from $l_s$ if and only if it represents a face of $\cl(G_{s-1})$ which intersects that union of open stars. Such a vertex $x$ itself lies in that union, therefore in $B_{s,u}\cup B_{s,v}$. To prove d) note that if $x$ is a vertex of $B_{s,u}\cap S_{s,v}$ then $x$ lies in the link $\lk_Kv$, hence it represents a face of $\cl(G_{s-1})$ contained in that link. By induction all vertices of that face are $l_{s-1}$-labeled $0$ or unlabeled hence $l_s(x)=0$ (since $x\in B_{s,u}$ it cannot remain unlabeled). This and a symmetric argument for $S_{s,u}\cap B_{s,v}$ proves d).

If $\tau\in \cl(G_{s-1})$ is a simplex with $l_{s-1}(\tau)=\{a\}$ then every vertex $x\in G_s$ which subdivides a face of $\tau$ will receive $l_s$-label $a$ or no label at all and therefore b), c) still hold for the simplices and edges of $\cl(G_s)$ contained within $\tau$. Now suppose that $l_{s-1}(\tau)=\{a,b\}$. Note that no simplex of $\cl(G_s)$ subdividing $\tau$ contains vertices $x,y$ with $l_s(x)=a$ and $l_s(y)=b$. Indeed, if $\tau_1$ and $\tau_2$ are the faces of $\tau$ in $\cl(G_{s-1})$ represented by $x$ and $y$ respectively, then $\tau_1\cap l_{s-1}^{-1}(a)\neq\emptyset$, $\tau_1\cap l_{s-1}^{-1}(b)=\emptyset$ and vice versa for $\tau_2$. But then neither $\tau_1\subseteq \tau_2$ nor $\tau_2\subseteq \tau_1$ hence $xy$ is not an edge in $\bary(\tau)$. Eventually we conclude that for every simplex of $\cl(G_s)$ subdividing $\tau$ the set of $l_s$-labels is either empty, or a singleton or one of $\{a,\frac{a+b}{2}\}$, $\{b,\frac{a+b}{2}\}$. This, together with the induction hypothesis, proves b) and c) in this case.

The existence of the partial labeling $l_s$ completes the proof of the lemma: every path from $B_{s,u}\cap S_{s,v}$ to $S_{s,u}\cap B_{s,v}$ in $G_s[B_{s,u}\cup B_{s,v}]$ passes through $l_s$-labeled vertices (by a)). In each step the label changes by at most $\frac{1}{2^s}$ (by c)) while the total change is $1$ (by d)). It means that the path requires at least $2^s$ steps. Of course there exists a path (e.g. the subdivision of the edge $uv$) of length exactly $2^s$.

\end{proof}

For a convenient reference let us summarize the proof of Theorem \ref{thm:explicit}.
\begin{proof}[Proof of Theorem \ref{thm:explicit}]
Fix $1\leq r<2^{s-2}$. Consider the subcomplexes $X_{s,v}$ of $\cl(G_s^r)$ defined in (\ref{eq:defofx}). By Proposition \ref{lemma:covering} they form a covering of $\cl(G_s^r)$ with nerve $K$. By Proposition \ref{lemma:intersect} every nonempty intersection of the $X_{s,v_i}$ is of the form
$$\cl((G_s[B_{s,v_1}\cap\cdots\cap B_{s,v_k}])^r).$$
Every such complex is contractible because Proposition \ref{prop:doch} and Lemma \ref{lemma:dismant} imply that the graph $(G_s[B_{s,v_1}\cap\cdots\cap B_{s,v_k}])^r$ is dismantlable. The equivalence $\cl(G_s^r)\htpyequiv K$ now follows from the nerve lemma \cite[15.21]{Koz}.
\end{proof}

\begin{remark}
The purpose of \cite{Doch} was to prove that for any complex $K$ and any connected, non-discrete graph $T$ there exists a graph $G$ with a homotopy equivalence
$$\cl(G^T)\htpyequiv K$$
where $(-)^T$ denotes the \emph{exponential graph} functor, the right adjoint to the categorical product $-\times T$ of graphs (see \cite[18.18]{Koz}). The idea was to use the graph $G_s$ and its subgraphs $G_s[B_{s,v}]$ to form a covering of $G_s^T$ (for $s$ depending on the diameter of $T$). Despite these similarities the author does not see a direct way to compare (up to homotopy) the complexes $\cl(G^r)$ of distance graph powers with any of the complexes $\cl(G^T)$.
\end{remark}

\section{Line graphs and edge subdivisions.}
\label{sect:line}
Let $S(G)$ denote the graph obtained from $G$ by subdividing every edge with one vertex. The graph $T(G)=S(G)^2$ is often called the \emph{total graph} of $G$. Recall that the \emph{line graph} $L(G)$ of $G$ is the incidence graph of the edges of $G$.

Write $V(S(G))$ as $\mathcal{V}\cup\mathcal{E}$ where $\mathcal{V}$ is the set of original vertices of $G$ and $\mathcal{E}$ is the set of subdividing vertices, one for each edge. Then we have isomorphisms
$$T(G)[\mathcal{V}]=G,\quad T(G)[\mathcal{E}]=L(G)$$
and we see that the inclusions
$$\cl(L(G))=\cl(T(G))[\mathcal{E}]\hookrightarrow\ \cl(T(G))\ \hookleftarrow \cl(T(G))[\mathcal{V}]=\cl(G)$$
make $\cl(T(G))$ a subcomplex of the join $\cl(G)\ast\cl(L(G))$.

Denote by $\te(G)$ the number of triangles in $G$. Then we have the following result.
\begin{theorem}
\label{thm:subdiv}
For any graph $G$ there is a homotopy equivalence
$$\cl(T(G))\htpyequiv \cl(G)\vee\bigvee^{\te(G)}S^2.$$
\end{theorem}
This is another way in which a given complex can be represented as a clique complex of a graph square up to homotopy and up to a number of $2$-spheres. As a byproduct of the proof method we also obtain the next result. Recall that $K^{(2)}$ denotes the $2$-dimensional skeleton of $K$.
\begin{theorem}
\label{cor:wedgess2}
For any non-discrete, connected graph $G$ 
$$\cl(L(G))\htpyequiv \cl(G)^{(2)}.$$
\end{theorem}
Both theorems depend on a simple classification.
\begin{lemma}
\label{lem:cliquesinsd}
Every maximal face in $\cl(T(G))$ is of one of the following forms:
\begin{itemize}
\item[a)] a maximal face of dimension at least $2$ in $\cl(G)$,
\item[b)] $\{v,e,w\}$ where $v,w$ are vertices of $G$ and $e=vw$,
\item[c)] $\{v,e_1,\ldots,e_k\}$ where $e_i$ are the edges incident with a vertex $v$ of $G$,
\item[d)] $\{e_1,e_2,e_3\}$ where $e_1,e_2,e_3$ are edges forming a triangle in $G$.
\end{itemize}
\end{lemma}
\begin{proof}
Let $\sigma$ be a maximal face in $\cl(T(G))$. If $\sigma$ contains at least three vertices of $\mathcal{V}$ then those vertices form a clique in $G$ and no edge is incident with all of them, so it cannot be extended by a vertex of $\mathcal{E}$. If $\sigma$ contains precisely two vertices $v,w$ of $\mathcal{V}$, then $e=vw$ is the only vertex of $\mathcal{E}$ adjacent to both of them. If $|\sigma\cap\mathcal{V}|=\{v\}$ then $\sigma$ must be of the form c). Eventually if $\sigma=\{e_1,\ldots,e_k\}\in\cl(L(G))$ then not all of $e_i$ are incident with a common vertex, but every two $e_i$, $e_j$ have a common vertex. This easily implies $k=3$.
\end{proof}

\begin{proof}[Proof of Theorem \ref{thm:subdiv}]
Consider the subcomplex $K\subset \cl(T(G))$ consisting of all faces of $\cl(T(G))$ which are not of the form $\{e_1,e_2,e_3\}$ for some three edges forming a triangle in $G$. Consider a matching on $K$ defined as follows
\begin{itemize}
\item the faces of $\cl(G)$ are unmatched,
\item for each edge $e=uv$ the faces $\{e\}$, $\{v,e\}$, $\{e,u\}$ and $\{v,e,u\}$ are unmatched,
\item for every face $\sigma\in K\cap\cl(L(G))$ of dimension at least one there exists a unique vertex $v\in G$ such that $\sigma\cup\{v\}$ is a face of $K$ (that vertex is the common end of the edges of $\sigma$). In such case match $\sigma$ with $\sigma\cup\{v\}$.
\end{itemize}
This is clearly an acyclic matching on $K$ in the sense of Definition \ref{acyclic}. Its critical faces form the subcomplex 
$$K'=\cl(G)\cup\{\{e\},\{v,e\},\{e,u\},\{v,e,u\}\ \textrm{for}\ e=uv\in\mathcal{E}\}.$$
This $K'$ easily collapses to $\cl(G)$, therefore also $K$ collapses to $\cl(G)$.

Now $\cl(T(G))$ arises from $K$ by attaching the $\te(G)$ cells $\{e_1,e_2,e_3\}$ for all triangles $\{v_1,v_2,v_3\}$ of $G$. The attaching map of every such cell is homotopic in $K$ to the boundary of the face $\{v_1,v_2,v_3\}$ of $\cl(G)$, therefore it is null-homotopic. It follows that $\cl(T(G))\htpyequiv K\vee \bigvee^{\te(G)}S^2 \htpyequiv \cl(G) \vee \bigvee^{\te(G)}S^2$.
\end{proof}

\begin{proof}[Proof of Theorem \ref{cor:wedgess2}]
Let $K$ be the subcomplex of $\cl(T(G))$ consisting of all faces $\sigma$ such that $|\sigma\cap\mathcal{V}|\leq 1$. Then $K$ is the union of $\cl(L(G))$ and simplices of the form \ref{lem:cliquesinsd}.c) for every $v\in\mathcal{V}$. For each $v$ the link $\lk_K(v)\subset\cl(L(G))$ is contractible (because it is a simplex) hence the removal of $v$ from $K$ does not change the homotopy type. It means that $K\htpyequiv\cl(L(G))$.

Let $K'\subset K$ be obtained from $K$ by removing the maximal faces $\{e_1,e_2,e_3\}$ corresponding to triangles of $G$. Then $K'$ is collapsible to the graph $S(G)$ by an acyclic matching argument identical to that used in \ref{thm:subdiv}, pairing $\sigma$ with $\sigma\cup\{v\}$ for any set $\sigma$ of at least two elements of $\mathcal{E}$ and their common endpoint $v$. Note that $S(G)$ and $G$ are homeomorphic as spaces.  

Now $K$ is recovered from $K'$ by attaching a $2$-face $\{e_1,e_2,e_3\}$ for every triangle $t$ of $G$. The attaching map is homotopic in $K'$ to the inclusion of $S(t)$ in $S(G)$. It follows that $K$ is homotopy equivalent to $G$ with a $2$-cell attached along every triangle. This is precisely $\cl(G)^{(2)}$. It follows that
$$\cl(L(G))\htpyequiv K\htpyequiv K'\cup\coprod^{t(G)}\Delta^2\Big/_\sim \htpyequiv \cl(G)^{(2)}.$$
\end{proof}

\begin{example}
The \emph{stable Kneser graph} $SG_{n,k}$ is a graph whose vertices are the $n$-element subsets of $\{1,\ldots,k+2n\}$ which do not contain two consecutive (in the cyclic sense) elements. One of the goals of \cite{Braun} is to calculate the homotopy types of the independence complexes $\ind(SG_{2,k})$. Since the complex $\ind(SG_{2,k})$ is exactly $\cl(L(\compl{C_{k+4}}))$, Theorem \ref{cor:wedgess2} identifies it, up to homotopy, with $\ind(C_{k+4})^{(2)}$. This explains why these space are homotopically at most two-dimensional, as stated in \cite[Thm.1.4]{Braun}.
\end{example}

\begin{remark}
From the two theorems of this section we immediately recover the result of \cite[Cor. 5.4]{LPV}, which is that the spaces $\cl(G)$, $\cl(T(G))$ and $\cl(L(G))$ have isomorphic fundamental groups.
\end{remark}

\section{Clique complexes of powers of cycles}
\label{sec:cyclepowers}

In this section we determine the homotopy types of the clique complexes of the graphs $C_n^r$, i.e. the powers of cycles. It follows from Proposition \ref{prop:largegirth} that for $1\leq r\leq\frac{n-1}{3}$ the complex $\cl(C_n^r)$ collapses to $\cl(C_n)\htpyequiv S^1$. On the other hand, for $r\geq\lfloor\frac{n}{2}\rfloor$ the complex $\cl(C_n^r)=\cl(K_n)$ is contractible. The intermediate values for some small pairs $n$, $r$ are shown in Section \ref{sect:tableofcyc}. The purpose of this section is to exhibit a systematic pattern in that table. It turns out to be best expressed in terms of the independence complexes of the complements of $C_n^r$. These results may also be interesting on their own right as one way of generalizing the calculation of Kozlov \cite{Koz2} of the homotopy types of $\ind(C_n)$.

For any pair of non-negative integers $n$, $k$ \emph{of opposite parity} and with $1\leq k\leq n-1$ let $T_{n,k}$ denote the graph obtained by connecting every vertex of the regular $n$-gon with the $k$ ``most opposite'' vertices. The notion of ``most opposite'' is well defined if $n$ and $k$ have opposite parity. For example, $T_{n,1}$ is the disjoint union of $\frac{n}{2}$ edges and examples of $T_{n,2}$ and $T_{n,3}$ are shown in Fig.\ref{fig:t9}a and Fig.\ref{fig:transf}a. To describe these graphs we are also going to use another parameter $r=r(n,k)=\frac{n-k-1}{2}$. Of course
\begin{equation*}
T_{n,k}=\compl{C_n^r}.
\end{equation*}
The graphs $T_{n,k}$ are called \emph{circular complete graphs} and form a subclass of \emph{circulant graphs}. The usual notation for $T_{n,k}=\compl{C^r_n}$ is 
$$K_{n/r+1} \quad \textrm{or}\quad C_n(r+1,\ldots,\lfloor\frac{n}{2}\rfloor)$$
but we will keep using the notation $T_{n,k}$ which is more intuitive for this application. For information about circulant and circular complete graphs and their independent sets see e.g. \cite{BPT,Ho2,Ho,Muz}.

We will identify the vertices of $T_{n,k}$ with $\zet/n$. Under this identification each vertex $i$ is connected to the vertices in the set
\begin{equation}
N_{T_{n,k}}(i)=\{i+r+1,\ldots,i+r+k\}\mod n.
\end{equation}
We are also going to need another auxiliary graph $S_{n,k}$. It is the induced subgraph of $T_{n,k}$ on the vertex set
\begin{equation}
V=\{1,\ldots,r\}\cup\{-1,\ldots,-r\}
\end{equation}
equipped \emph{additionally} with the edges $(-i,j)$ for all pairs $i,j\in\{1,\ldots,k-1\}$ such that $i+j\leq k$. For examples of $S_{n,k}$ see Fig.\ref{fig:t9}b and Fig.\ref{fig:transf}b.

The main results that lead to the calculation of $\cl(C_n^r)$ are the following propositions. Recall that $\susp$ denotes the unreduced suspension.

\begin{proposition}
\label{prop1}
If $n\geq 3k-1$ then
\begin{equation*}
\ind(T_{n,k})\htpyequiv\susp\ind(S_{n,k}).
\end{equation*}
\end{proposition}

\begin{proposition}
\label{prop2}
If $n\geq 3k+3$ then
\begin{equation*}
\ind(S_{n,k})\htpyequiv\susp\ind(T_{n-2(k+1),k}).
\end{equation*}
\end{proposition}

\begin{figure}
\begin{tabular}{cc}\includegraphics[scale=0.5]{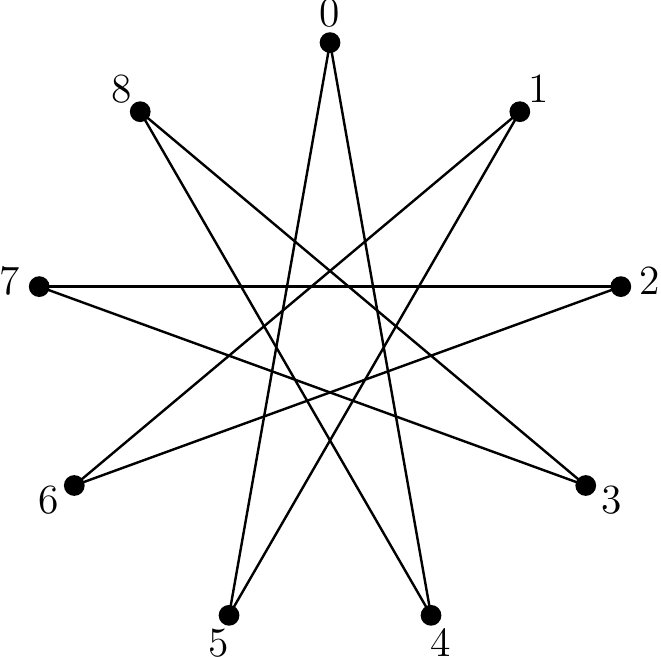} &\includegraphics[scale=0.5]{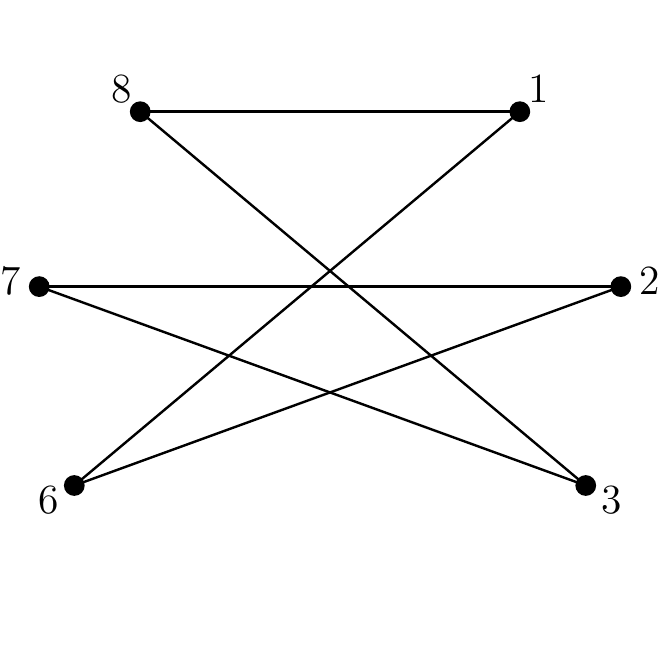} \\
a) & b) \end{tabular}
\caption{a) $T_{9,2}=\compl{C_9^3}$. b) $S_{9,2}$.}
\label{fig:t9}
\end{figure}

\begin{example}
Consider the special case $k=2$. We have $T_{n,2}=C_n$ for every odd $n\geq 3$ and $S_{n,2}=C_{n-3}$ for every odd $n\geq 7$. The previous two propositions thus combine to the statement \begin{equation*}\ind(C_m)\htpyequiv\susp\ind(C_{m-3}) \qquad \textrm{ for all } m\geq 6.\end{equation*}
Moreover $\ind(C_3)\homeo S^0\vee S^0$, $\ind(C_4)\htpyequiv S^0$ and $\ind(C_5)=\cl(\compl{C_5})=\cl(C_5)\htpyequiv S^1$ so it follows by induction that for all $m\geq 1$
\begin{equation*}
\ind(C_{3m})\htpyequiv S^{m-1}\vee S^{m-1}, \quad \ind(C_{3m+1})\htpyequiv S^{m-1}, \quad  \ind(C_{3m+2})\htpyequiv S^{m}.
\end{equation*}
This was first established by Kozlov \cite{Koz2} and then reproved in a number of ways.
\end{example}

\begin{corollary}
\label{cor:susp2}
If $n\geq 3k+3$ then 
\begin{equation*}
\ind(T_{n,k})\htpyequiv\susp^2\ind(T_{n-2(k+1),k}).
\end{equation*}
\end{corollary}

\begin{corollary}
\label{cor:indepresult}
For any $1\leq k\leq n-1$, with $k$ and $n$ of opposite parity, we have
\begin{equation*}
\ind(T_{n,k})\htpyequiv\left\{ 
\begin{array}{ll}
\bigvee^k S^{2l} & \mathrm{if}\ n=(2l+1)(k+1) \\
S^{2l+1}         & \mathrm{if}\ (2l+1)<\frac{n}{k+1}<(2l+3) \\
\end{array}
\right. \mathrm{for\ some}\ l\geq 0.
\end{equation*}
\end{corollary}
\begin{proof}
First we establish the result when $k+1\leq n\leq 3k+2$. If $n=k+1$ then
\begin{equation*}
\ind(T_{k+1,k})=\ind(K_{k+1})\homeo\bigvee^k S^0.
\end{equation*}
Now suppose that $k+2\leq n\leq 3k+2$. These inequalities imply that $r\geq 1$, $3r+1\leq n$ and $1<\frac{n}{k+1}<3$ so $l=0$. Since $n\geq 4$ by Proposition \ref{prop:largegirth} we get
\begin{equation*}
\ind(T_{n,k})=\cl(C_n^r)\htpyequiv S^1=S^{2l+1}
\end{equation*}
as required. For $n\geq 3k+3$ the result follows by induction using Corollary \ref{cor:susp2} because every increase of $n$ by $2(k+1)$ adds a double suspension to the homotopy type.
\end{proof}

These results can be transformed into statements about $\cl(C_n^r)$ by a straightforward calculation. Corollary \ref{cor:susp2} translates into:
\begin{corollary}
\label{cor:suspforcliques}
For any $\frac{n}{3}\leq r< \frac{n}{2}$
$$\cl(C_n^r)\htpyequiv\susp^2\cl(C_{4r-n}^{3r-n})=\susp^2\cl(C_{n-2\cdot (n-2r)}^{r-1\cdot (n-2r)}).$$
\end{corollary}
It follows that in the $(n,r)$-chart of the complexes $\cl(C_n^r)$ (see Section \ref{sect:tableofcyc}) the double suspension operator $\susp^2$ acts always along the lines of slope $(2,1)$. The translation of Corollary \ref{cor:indepresult} is:

\begin{corollary}
\label{cor:cliqueresult}
For any $n\geq 3$ and $0\leq r<\frac{n}{2}$ we have
\begin{equation*}
\cl(C_n^r)\htpyequiv\left\{ 
\begin{array}{ll}
\bigvee^{n-2r-1} S^{2l} & \mathrm{if}\ r=\frac{l}{2l+1}n \\
S^{2l+1}         & \mathrm{if}\ \frac{l}{2l+1}n<r<\frac{l+1}{2l+3}n \\
\end{array}
\right. \mathrm{for\ some}\ l\geq 0.
\end{equation*}
\end{corollary}

\begin{remark} The relevant value of $l$ for each pair $(n,r)$ is given by
\begin{equation*}
l=\lfloor\frac{r}{n-2r}\rfloor
\end{equation*}
\end{remark}

It remains to prove Propositions \ref{prop1} and \ref{prop2}. Our tool to analyze the homotopy types of $\ind(T_{n,k})$ and $\ind(S_{n,k})$ are the \emph{star clusters} introduced by J.Barmak \cite{Bar}. Let us recall the main result of that work.

\begin{theorem}[Barmak, \cite{Bar}]
\label{barmakthm}
Suppose $v$ is a non-isolated vertex of $G$ which does not belong to any triangle. Let $K$ be the subcomplex of $\ind(G)$ defined as
\begin{equation}
K=\st(v)\cap\bigcup_{w\in N_G(v)} \st(w)
\end{equation}
where all stars are taken in $\ind(G)$. Then there is a homotopy equivalence $\ind(G)\htpyequiv \susp K$.
\end{theorem}

In the proofs of Propositions \ref{prop1} and \ref{prop2} we are going to choose a vertex $v$ as in the theorem and identify the subcomplex $K$ with the independence complex of some graph using the following technical lemma.

\begin{lemma}
\label{technicallemma}
Let $v_1,v_2,\ldots,v_{2d}$ be a sequence of (not necessarily distinct) vertices of $G$ such that every consecutive $d+1$ vertices $v_i,v_{i+1},\ldots,v_{i+d}$ are pairwise distinct for $i=1,\ldots,d$. 

Let $L$ be the subcomplex of $\ind(G)$ consisting of those faces $\sigma$ which satisfy the condition
\begin{equation}
\label{strangecondition}
\{v_s,v_{s+1},\ldots,v_{s+d-1}\}\cap\sigma=\emptyset\ \textrm{for some}\ s\in\{1,2,\ldots,d+1\}.
\end{equation}

Then $L$ is isomorphic with the complex $\ind(H)$, where $H$ is a graph obtained from $G$ by adding the edges 
$$(v_i,v_j)\ \textrm{for all}\ 1\leq i\leq d,\ d+1\leq j\leq 2d,\ \textrm{such that}\ j-i\leq d.$$
\end{lemma}
\begin{proof}
Let $\sigma$ be a face of $\ind(G)$ which satisfies (\ref{strangecondition}) for some $s$. Since every pair $(v_i,v_j)$ with $1\leq i\leq d$, $d+1\leq j\leq 2d$ and $j-i\leq d$ has at least one of its elements in $\{v_s,v_{s+1},\ldots,v_{s+d-1}\}$, the face $\sigma$ cannot contain both elements $v_i$ and $v_j$ simultaneously. It means that $\sigma$ determines an independent set in $H$. Conversely, if $\sigma$ is a face of $\ind(H)$ then define $s$ by the formula
\begin{equation}
s=1+\max\{1\leq i\leq d: v_i\in\sigma\}
\end{equation}
(where $\max\emptyset=0$). One easily checks that $s$ and $\sigma$ satisfy (\ref{strangecondition}).
\end{proof}

\begin{figure}
\begin{tabular}{ccc}\includegraphics[scale=0.5]{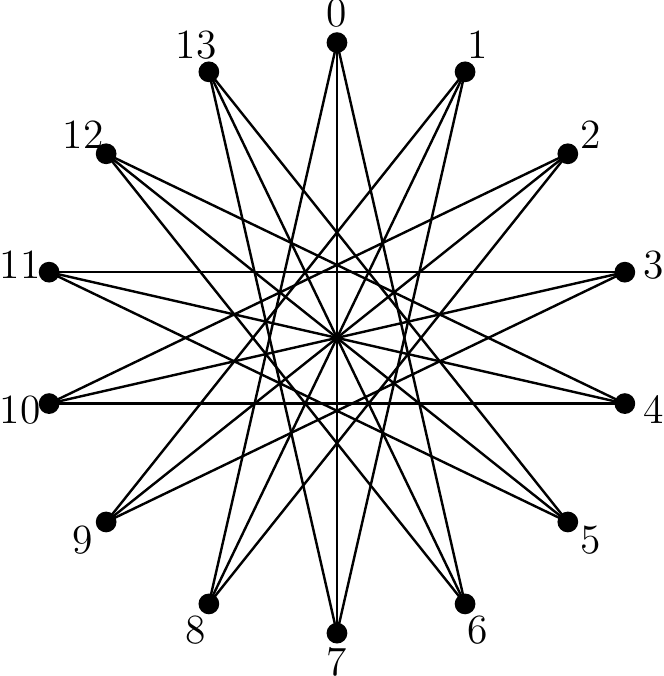} &\includegraphics[scale=0.5]{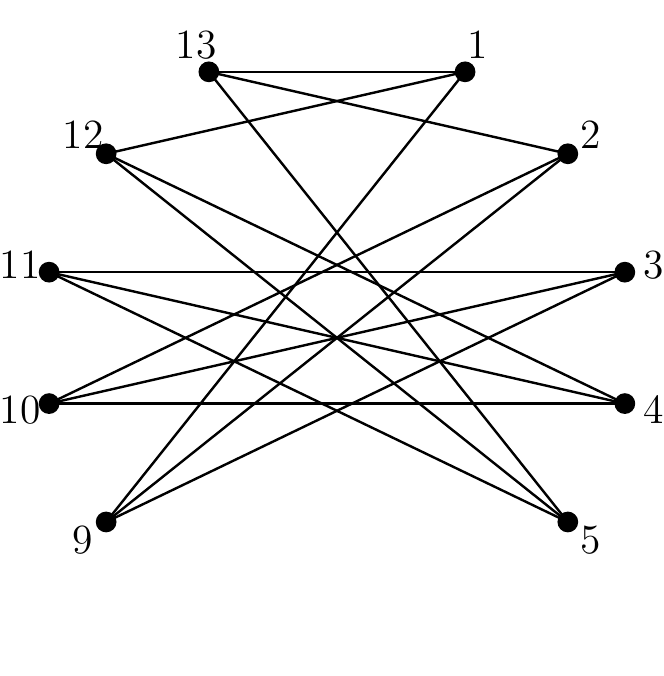} & \includegraphics[scale=0.5]{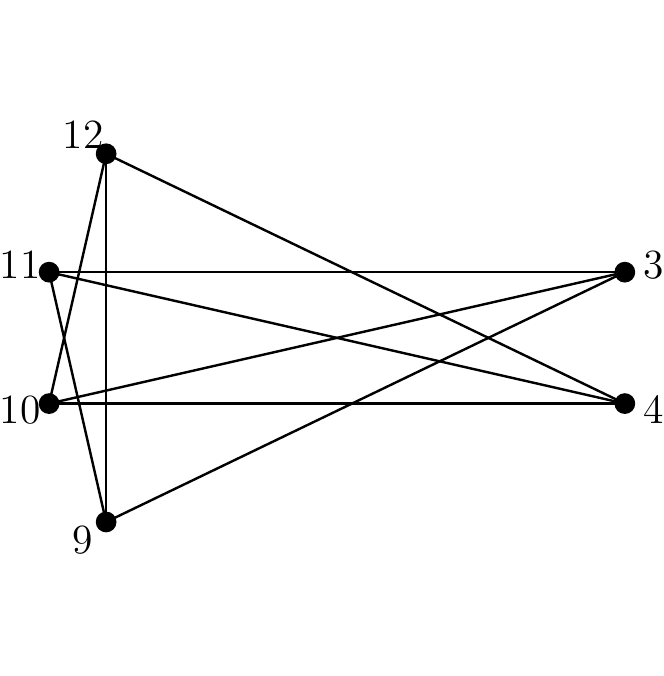}\\
a) & b) & c)\end{tabular}
\caption{a) $T_{14,3}=\compl{C_{14}^5}$. b) $S_{14,3}$. c) A graph $H$ obtained in the proof of Prop.\ref{prop2}, isomorphic to $T_{6,3}$.}
\label{fig:transf}
\end{figure}

\begin{proof}[Proof of Proposition \ref{prop1}.]
First note that the assumption $n\geq 3k-1$ is equivalent with $r\geq k-1$. Consider the vertex $0$ of $T_{n,k}$. Its neighbours in $T_{n,k}$ are the vertices of 
\begin{equation*}
\neib{0}=\{r+1,r+2,\ldots,r+k\}.
\end{equation*}
No two of those vertices are adjacent because their distances along the circle are at most $k-1\leq r$, so $0$ is not in any triangle. By Theorem \ref{barmakthm} $\ind(T_{n,k})\htpyequiv\susp K$ where $K$ is the subcomplex of $\ind(T_{n,k})$ given by
\begin{equation*}
K=\st(0)\cap\bigcup_{w\in \neib{0}} \st(w).
\end{equation*}
Note that $K$ is in fact a subcomplex of $\ind(T_{n,k}[V])$, where $V=\{1,\ldots,r\}\cup\{-1,\ldots,-r\}$ is the set of vertices non-adjacent to $0$ in $T_{n,k}$. The complex $K$ consists precisely of those independent sets $\sigma$ in $T_{n,k}[V]$ for which there exists a vertex $w\in\neib{0}$ such that $\sigma\cup\{w\}$ is an independent set in $T_{n,k}$ or, in other words, such that $\sigma\cap\neib{w}=\emptyset$. If $w=r+j$ for $1\leq j\leq k$ then (see also Fig.\ref{fig:schem}a)
\begin{equation*}
\neib{r+j}\cap V=\{-(k-j),\ldots,-1\}\cup\{1,\ldots,j-1\}.
\end{equation*}
It follows that $\sigma$ is a face of $K$ if and only if it is an independent set of $T_{n,k}[V]$ such that
\begin{equation*}
\{-(k-j),\ldots,-1,1,\ldots,j-1\}\cap\sigma=\emptyset\ \textrm{for some}\ j\in\{1,2,\ldots,k\}.
\end{equation*}
We can now apply Lemma \ref{technicallemma} with $d=k-1$, $G=T_{n,k}[V]$ and $(v_1,\ldots,v_{2d})=(-(k-1),\ldots,-1,1,\ldots,k-1)$, where all the vertices in the last sequence are distinct. The graph $H$ obtained in the lemma is $S_{n,k}$ because the additional edges are precisely $(-(k-1),1)$, $(-(k-2),1)$, $(-(k-2),2)$, etc., as in the definition of $S_{n,k}$. Therefore $\ind(T_{n,k})\htpyequiv\susp K=\susp \ind(H)=\susp\ind(S_{n,k})$.
\end{proof}

\begin{figure}
\begin{tabular}{cc}\includegraphics[scale=0.53]{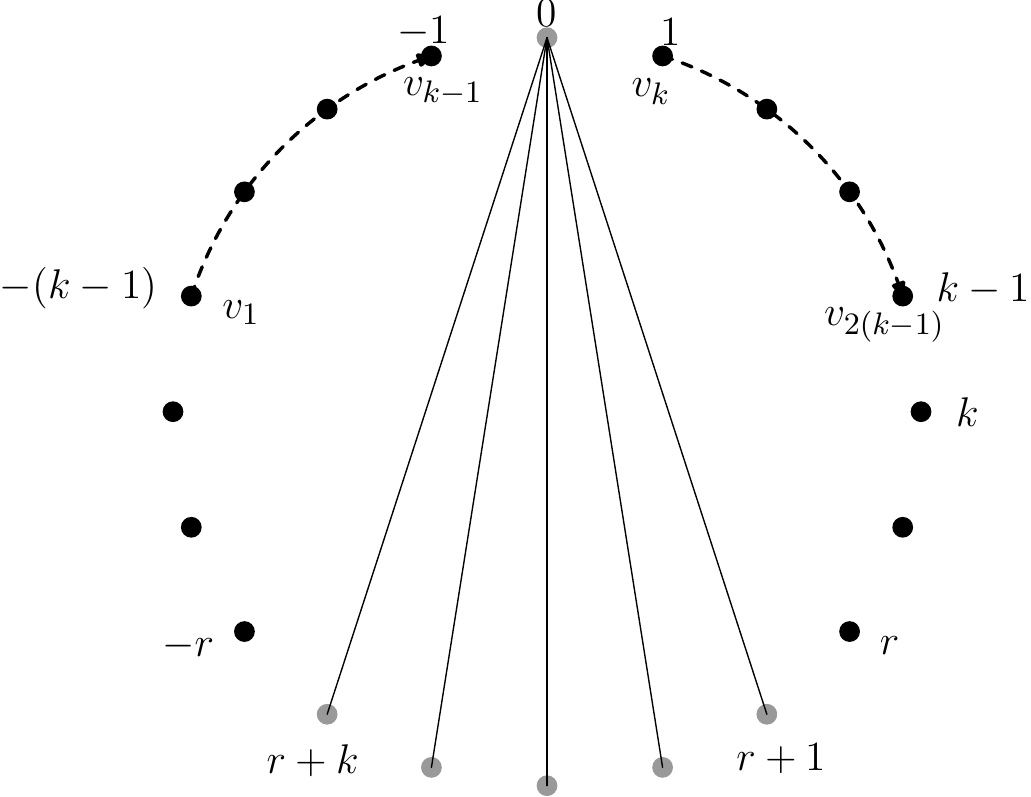} &\includegraphics[scale=0.53]{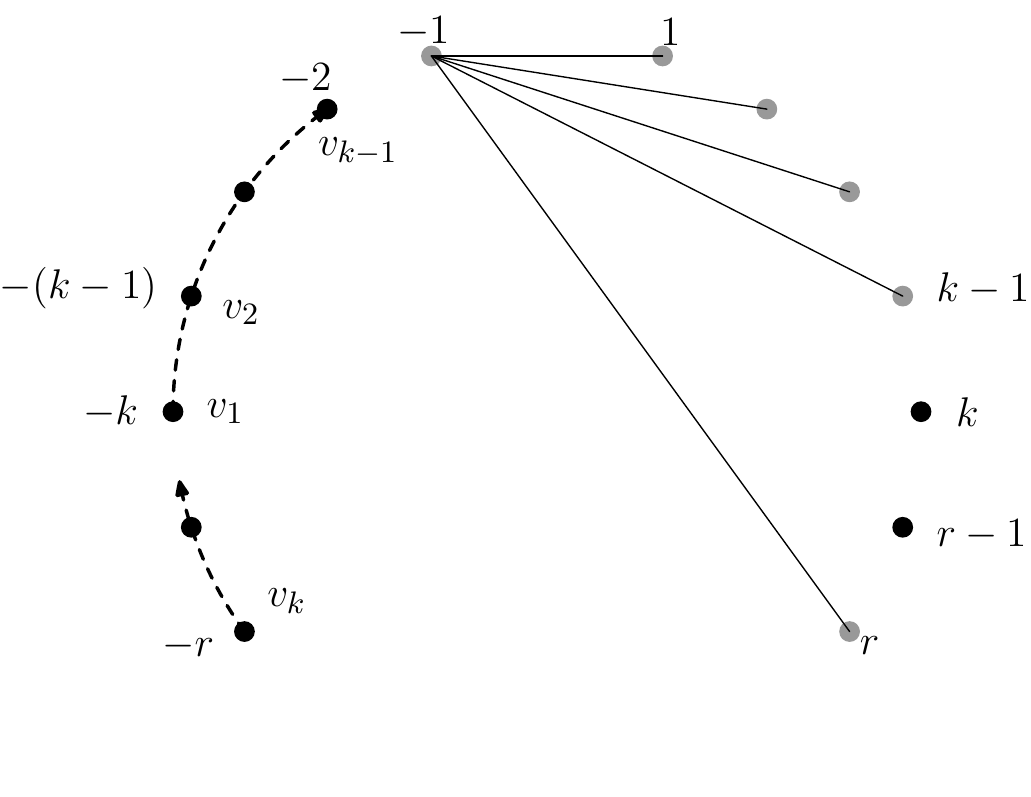}\\a) & b)\end{tabular}
\caption{Schematics for the proofs of \ref{prop1} and \ref{prop2}. The shaded vertices are in $N_G[v]$ and will be removed. Additional edges will be added between vertices marked with dashed lines.}
\label{fig:schem}
\end{figure}

\begin{proof}[Proof of Proposition \ref{prop2}.]
The assumption $n\geq 3k+3$ is equivalent with $r\geq k+1$. We apply the same strategy as before with respect to the vertex $(-1)$. Its neighbours in $S_{n,k}$ are
\begin{equation*}
\neib{-1}=\{1,\ldots,k-1\}\cup\{r\}.
\end{equation*}
No two of these vertices are adjacent, so $(-1)$ is not in any triangle. Exactly as before we obtain that $\ind(S_{n,k})\htpyequiv\susp K$ where $K$ is a subcomplex of $\ind(S_{n,k}[V])$ where
\begin{equation*}
V=\{k,\ldots,r-1\}\cup\{-r,\ldots,-2\}
\end{equation*}
(see Fig.\ref{fig:schem}b) is the set of vertices of $S_{n,k}$ non-adjacent to $(-1)$. Note that both sets in the above union are nonempty. The complex $K$ consists of those faces $\sigma$ of $\ind(S_{n,k}[V])$ for which there exists a $w\in\neib{-1}$ such that $\sigma\cap\neib{w}=\emptyset$. Note that
\begin{eqnarray*}
\neib{r}\cap V&=&\{-k,\ldots,-2\}   \\
\neib{1}\cap V&=&\{-(k-1),\ldots,-2\}\cup\{-r\}   \\
\neib{2}\cap V&=&\{-(k-2),\ldots,-2\}\cup\{-r,-r+1\}   \\
&\cdots&\\
\neib{k-1}\cap V&=&\{-r,-r+1,\ldots,-r+(k-2)\}.
\end{eqnarray*}
In the sequence
\begin{equation*}
S=(-k,\ldots,-2,-r,\ldots,-r+(k-2))
\end{equation*}
of length $2(k-1)$ every $k$ consecutive vertices are pairwise distinct. Because of the cyclic behaviour it is enough to check this for the subsequence $(-k,\ldots,-2,-r)$, where it boils down to the inequality $-r<-k$ which follows from $r\geq k+1$. 

By Lemma \ref{technicallemma} the complex $K$ is therefore homotopy equivalent to $\ind(H)$, where $H$ arises from $S_{n,k}[V]$ by adding the edges $(-(k-i),-r+j)$ for all $0\leq j \leq i \leq k-2$. It remains to identify this graph $H$ with $T_{n-2(k+1),k}$. This can be best seen geometrically (cf. Fig.\ref{fig:transf}c, Fig.\ref{fig:schem}b). The graph $H$ differs from $T_{n,k}$ by the removal of $2(k+1)$ vertices $\{-1,0,\ldots,k-1\}\cup\{r,r+1,\ldots,r+k\}$ which form two ``gaps'' of length $k+1$ each. Note that the vertices not in $S$ are not affected at all by the construction, so their neighbourhoods in $T_{n,k}$ and $H$ coincide. The vertices in $S$ are located at most $k-1$ steps from the boundaries of the gaps and for them the missing connections are provided by the extra edges in $H$, so that the neighbours of each vertex of $S$ form a contiguous block of length $k$ in the cyclic ordering of vertices in $H$ inherited from $T_{n,k}$ and we again have a circular complete graph $T_{n-2(k+1),k}$. This identification completes the proof of the proposition: $\ind(S_{n,k})\htpyequiv \susp K\homeo \susp \ind(H)\homeo \susp\ind(T_{n-2(k+1),k})$.
\end{proof}

\section{Appendix: The table of clique complexes of cycle powers}
\label{sect:tableofcyc}
The table presents the  homotopy types of some initial clique complexes $\cl(C_n^r)$. The entries below the shaded area are all $S^1$ by Proposition \ref{prop:largegirth} and the entries above it are all $\ast$ (a contractible space). The arrows show the action of the double suspension operator $\Sigma^2$ of Corollary \ref{cor:suspforcliques}.
\clearpage
\small{
\begin{table}[!ht]
\begin{adjustwidth}{-1in}{-1in}
\begin{minipage}[b]{1 \linewidth}\centering
\begin{tabular}{l|c|c|c|c|c|c|c|c|c|c|c|c}
$r=$   &$0$&  $1$    &   $2$   &   $3$   &   $4$   &   $5$   &   $6$   &  $7$   &  $8$  &  $9$  &  $10$  &  $11$ \\\hline
$C_3$  &\xx{$\x{2}{0}$} &$\ast$&$\ast$   & $\ast$  & $\ast$  &         &         &        &       &       &        &       \\\hline
$C_4$  &\xx{$\x{3}{0}$} &{$S^1$}&$\ast$&$\ast$   & $\ast$  &         &         &        &       &       &        &       \\\hline
$C_5$  &\xx{$\x{4}{0}$} &{$S^1$}&$\ast$&$\ast$   &$\ast$   &         &         &        &       &       &        &       \\\hline
$C_6$  &\xx{$\x{5}{0}$} &{$S^1$}&\xx{$S^2$}&$\ast$&$\ast$   & $\cdots$        &         &        &       &       &        &       \\\hline
$C_7$  &\xx{$\x{6}{0}$} &{$S^1$}&{$S^1$}&$\ast$& $\ast$  &         &         &        &       &       &        &       \\\hline
$C_8$  &\xx{$\x{7}{0}$} &{$S^1$}  &{$S^1$}  &\xx{$S^3$}&$\ast$   &         &         &        &       &       &        &       \\\hline
$C_9$  &\xx{$\x{8}{0}$} &\nnn{$S^1$}{1}  &{$S^1$}  &\xx{$\x{2}{2}$}&$\ast$  &         &         &        &       &       &        &       \\\hline
$C_{10}$&\xx{$\x{9}{0}$}&{$S^1$}  &{$S^1$}  &{$S^1$}  &\nnn{\xx{$S^4$}}{3}&         &         &        &       &       &        &       \\\hline
$C_{11}$&\xx{$\x{10}{0}$}&         &         &         &\xx{$S^3$}&         &         &        &       &       &        &       \\\hline
$C_{12}$&\xx{$\x{11}{0}$}&         &   $\cdots$ &         &\nnn{\xx{$\x{3}{2}$}}{2}&\xx{$S^5$}&         &        &       &       &        &       \\\hline
$C_{13}$&\xx{$\x{12}{0}$}&         &         &         &         &\xx{$S^3$}  &         &        &       &       &        &       \\\hline
$C_{14}$&\xx{$\x{13}{0}$}&         &         &         &         &\xx{$S^3$}  &\nnn{\xx{$S^6$}}{33}   &        &       &       &        &       \\\hline
$C_{15}$&\xx{$\x{14}{0}$}&         &         &         &         &\xx{$\x{4}{2}$}&\xx{$\x{2}{4}$}&      &       &       &        &       \\\hline
$C_{16}$&\xx{$\x{15}{0}$}&         &         &         &         &         &\xx{$S^3$}&\xx{$S^7$}&       &       &        &       \\\hline
$C_{17}$&\xx{$\x{16}{0}$}&         &         &         &         &         &\xx{$S^3$}&\xx{$S^5$}&       &       &        &       \\\hline
$C_{18}$&\xx{$\x{17}{0}$}&         &         &         &         &         &\xx{$\x{5}{2}$}&\xx{$S^3$}&\xx{$S^8$}&       &        &       \\\hline
$C_{19}$&\xx{$\x{18}{0}$}&         &         &         &         &         &         &\xx{$S^3$}&\xx{$S^5$}&       &        &       \\\hline
$C_{20}$&\xx{$\x{19}{0}$}&         &         &         &         &         &         &\xx{$S^3$}&\nnn{\xx{$\x{3}{4}$}}{22}&\xx{$S^9$}&       &       \\\hline
$C_{21}$&\xx{$\x{20}{0}$}&         &         &         &         &         &         &\xx{$\x{6}{2}$}&\xx{$S^3$}&\xx{$\x{2}{6}$}&     &      \\\hline
$C_{22}$&\xx{$\x{21}{0}$}&         &         &         &         &         &         &        &\xx{$S^3$}&\xx{$S^5$}&\xx{$S^{10}$}& \\\hline
$C_{23}$&\xx{$\x{22}{0}$}&         &         &         &         &         &         &        &\nnn{\xx{$S^3$}}{11}&\xx{$S^3$}&\xx{$S^{7}$} & \\\hline
$C_{24}$&\xx{$\x{23}{0}$}&         &         &         &         &         &         &        &\xx{$\x{7}{2}$}&\xx{$S^3$}&\xx{$S^5$}  &\xx{$S^{11}$}\\\hline
$C_{25}$&\xx{$\x{24}{0}$}&         &         &         &         &         &         &        &       &\xx{$S^3$}  &\xx{$\x{4}{4}$}&\xx{$S^7$}\\\hline
\end{tabular}
\end{minipage}

\begin{tikzpicture}[overlay]
\draw[->,line width=1pt] (1.east) -- (11.west);
\draw[->,line width=1pt] (2.east) -- (22.west);
\draw[->,line width=1pt] (3.east) -- (33.west);
\end{tikzpicture}\end{adjustwidth}\end{table}

\begin{table}[!ht]
\begin{tabular}{l|c|c|c|c|c|c|c}
$r=$    &  $8$    &   $9$   &   $10$  &   $11$  &   $12$  &   $13$  &  $14$   \\\hline
$C_{26}$&         &\xx{$S^3$}&\xx{$S^3$}&\xx{$S^5$}  &\xx{$S^{12}$}&         &         \\\hline
$C_{27}$&         &\xx{$\x{8}{2}$}&\xx{$S^3$}&\xx{$S^5$}&\xx{$\x{2}{8}$}&        &         \\\hline
$C_{28}$&         &         &\xx{$S^3$}&\xx{$S^3$}&\xx{$\x{3}{6}$}&\xx{$S^{13}$}&         \\\hline
$C_{29}$&         &         & \xx{$S^3$}  &\xx{ $S^3$ } &\xx{ $S^5$}   & \xx{$S^9$}   &         \\\hline
$C_{30}$&         &         &\xx{$\x{9}{2}$}&\xx{ $S^3$}  &\xx{$\x{5}{4}$}&\xx{$S^7$ }  &\xx{$S^{14}$} \\\hline
\end{tabular}
\end{table}
}


\end{document}